\documentclass[draft]{amsart}

\usepackage{amssymb}
\usepackage{mathrsfs}
\usepackage{accents}
\usepackage{listliketab}
\usepackage{stmaryrd}

\def\sp{\operatorname{sp}}

\renewcommand\epsilon{\varepsilon}

\def \<{\langle}
\def \>{\rangle}

\def \supp {\operatorname{supp}}

\def \((  {(\!(}
\def \)) {)\!)}

\DeclareMathSymbol{\precequ}{\mathrel}{symbols}{"16}
\DeclareMathSymbol{\succequ}{\mathrel}{symbols}{"17}

\newtheorem{theorem}{Theorem}[section]
\newtheorem{lemma}[theorem]{Lemma}
\newtheorem{prop}[theorem]{Proposition}

\theoremstyle{definition}

\theoremstyle{remark}

\def \fd {{\mathfrak d}}

\let\oldi\i
\let\oldj\j
\renewcommand\i{\relax\ifmmode{\boldsymbol{i}}\else\oldi\fi}
\renewcommand\j{\relax\ifmmode{\boldsymbol{j}}\else\oldj\fi}

\renewcommand\leq{\leqslant}
\renewcommand\geq{\geqslant}
\renewcommand\preceq{\preccurlyeq}

\renewcommand\le{\leq}
\renewcommand\ge{\geq}

\DeclareMathAlphabet{\mathbf}{OML}{cmm}{b}{it}

\DeclareFontFamily{U}{fsy}{}
\DeclareFontShape{U}{fsy}{m}{n}{<->s*[.9]psyr}{}
\DeclareSymbolFont{der@m}{U}{fsy}{m}{n}
\DeclareMathSymbol{\der}{\mathord}{der@m}{182}

\DeclareSymbolFont{der@m}{U}{fsy}{m}{n}
\DeclareMathSymbol{\derdelta}{\mathord}{der@m}{100}

\DeclareSymbolFont{imag@m}{OT1}{cmr}{m}{ui}
\DeclareMathSymbol{\imag}{\mathord}{imag@m}{105}

\DeclareFontFamily{OMS}{smallo}{}
\DeclareFontShape{OMS}{smallo}{m}{n}{<->s*[.65]cmsy10}{}
\DeclareSymbolFont{smallo@m}{OMS}{smallo}{m}{n}
\DeclareMathSymbol{\smallo}{\mathord}{smallo@m}{79}

\DeclareFontFamily{OMS}{largerdot}{}
\DeclareFontShape{OMS}{largerdot}{m}{n}{<->s*[.8]cmsy10}{}
\DeclareSymbolFont{largerdot@m}{OMS}{largerdot}{m}{n}
\DeclareMathSymbol{\largerdot}{\mathord}{largerdot@m}{15}

\DeclareMathSymbol{\llambda}{\mathord}{der@m}{108}
\DeclareMathSymbol{\rrho}{\mathord}{der@m}{114}

\newcommand{\equationqed}[1]{\[\pushQED{\qed}#1 \qedhere\popQED\]\let\qed\relax}
\newcommand{\alignqed}[1]{\begin{align*}\pushQED{\qed} #1 \qedhere\popQED\end{align*}\let\qed\relax}

\def \No{\text{{\bf No}}}

\begin{document}
\title{Truncation Structures}

\author[van den Dries]{Lou van den Dries}
\address{Department of Mathematics\\
University of Illinois at Urbana-Cham\-paign\\
Urbana, IL 61801\\
U.S.A.}
\email{vddries@illinois.edu}

\date{October 2025. Funding information to disclose: None} 

\begin{abstract} We characterize intrinsically the truncation structures on valued fields arising from
embeddings into Hahn fields with truncation closed image.
\end{abstract}

\maketitle

%\qquad \qquad \qquad \qquad {\bf TRUNCATION STRUCTURES}

\bigskip\noindent
\section{Introduction}

\medskip\noindent
Truncation of series in Hahn fields is surprisingly robust under various operations, as was first noticed and exploited by Mourgues and Ressayre~\cite{MR} for certain kinds of Hahn fields; see also \cite{D} and \cite{F} for general Hahn fields. An isomorphism of a valued field $F$ onto a truncation closed subfield of a Hahn field induces a truncation structure on $F$. Our goal is to give an intrinsic characterization of such truncation structures. 
This is done by (T1)--(T8) below and Theorem~\ref{t1}. Our characterization is of a first-order nature, except for one part, (T5),  that mentions well-orderings.

Below we use notations and terminology from \cite{D}. See also \cite[Section 8]{D} for further background on truncation and connections to o-minimality and to Conway's field $\No$ of surreal numbers.  

As to other work on ``axiomatizing''  truncation, \cite[Section 3.2]{C} contains partial results in this direction. 
One reviewer alerted me to \cite{FKK}, which also gives an intrinsic characterization of truncation closed embeddability. It is more complicated than the present treatment. In particular, ``T5 is a more sensible explanation of the non-first-order-aspect'' to quote another referee, who also remarks that most of \cite{FKK} easily translates into the setting of the present paper. 
%Roughly speaking, that characterization is more of a ``second order" nature.  

\medskip\noindent 
Let $F$ be a field equipped with a (Krull) valuation $v:F\to \Gamma\cup\{\infty\}$ (so $\Gamma$ is an ordered abelian group\footnote{Our convention here is that the (translation invariant) ordering of $\Gamma$ is total.}, additively written) and with a lift $C$ of the residue field, that is,  $C$ is a subfield of $F$, $C\subseteq \mathcal{O}:= \{f\in F:\ vf\ge 0\}$, the valuation ring of $F$, and $C$ is mapped bijectively onto the residue field of $\mathcal{O}$ by the residue map.  Let $\alpha, \beta, \gamma$ range over elements of $\Gamma$. 

A ``universal'' object of this kind is the Hahn field $C((t^\Gamma))$.  We represent an element $f$ of $C((t^\Gamma))$ as a formal sum $\sum_{\gamma}c_{\gamma}t^\gamma$, with coefficients $c_\gamma\in C$. The {\em support\/} of such $f$ is the wellordered subset $\{\gamma:\ c_{\gamma}\ne 0\}$ of $\Gamma$, and we set  $v_t(f):=\min \supp f$ if $f\ne 0$, and
$v_t(0):= \infty$, which gives us the Hahn valuation $v_t: C((t^\Gamma))\to \Gamma\cup \{\infty\}$. 
Now $C((t))$ has in addition a
truncation operation: the 
{\em $\alpha$-truncation\/} of such $f$ is the element
$f|_{\alpha}:=\sum_{\gamma<\alpha} c_\gamma t^\gamma$ of $C((t^\Gamma))$. Call a subset $S$ of
$C((t^\Gamma))$ {\em truncation closed\/} if $f|_\alpha\in S$ for all $f\in S$ and all $\alpha$.

The field $F$ comes equipped with the valuation $v: F\to \Gamma\cup\{\infty\}$ and the subfield $C$. Accordingly, a map
$e: F\to C((t^\Gamma))$ is said to be an {\em embedding\/} if $e$ is a field embedding that is the identity on $C$ such that $v(f)=v_t\big(e(f)\big)$ for all $f\in F$ (so $e$ is in particular a valued field embedding).
We are interested in the case where such an embedding has truncation closed image. 

Let $e: F \to C((t^\Gamma))$ be an
embedding  with truncation closed image. Then $e$ induces what we call a
{\em truncation operation} 
$$(f,\alpha)\mapsto f|_{\alpha}\ :\  F\times \Gamma\to F$$
on $F$ by the requirement $e(f|_{\alpha})=e(f)|_{\alpha}$
for all $f\in F$ and all $\alpha$. This operation has the following eight properties. First, for all $f,g\in F$, $c\in C$, and $\alpha, \beta$,
\begin{enumerate}
\item[(T1)] $v\big(f-(f|_{\alpha})\big)\ge \alpha$;
\item[(T2)] $v(f)\ge \alpha\ \Rightarrow\ f|_{\alpha}=0$;
\item[(T3)] $\beta> \alpha\ \Rightarrow\ (f|_{\alpha})|_{\beta}=f|_{\alpha}$;
\item[(T4)]  $(f+g)|_{\alpha}=(f|_{\alpha})+(g|_{\alpha})$ and $(cf)|_{\alpha}=c\cdot(f|_{\alpha})$.
%\item $\Gamma(f):=\{\gamma\in \Gamma:\ v\big(f-(f|_{\gamma})\big)=\gamma\}$ is wellordered.
\end{enumerate} 
The next two properties involve the sets 
$$\sp(f)\ :=\ \{\gamma:\ v\big(f-(f|_{\gamma})\big)=\gamma\}, \qquad(f\in F).$$ 
For $f\in C((t^\Gamma))$ we have $\supp f=\{\gamma:\ v_t(f-f|_{\gamma})=\gamma\}$, so for all $f, g\in F$ and $\gamma$: \begin{enumerate}
\item[(T5)]  $ \sp(f)$ is wellordered;
\item[(T6)] $\sp(f)+\sp(g)< \gamma\ \Longrightarrow\ \sp(fg)<\gamma$.
\end{enumerate}
Moreover, there is for each $\gamma$ a unique element
$\tau^\gamma$ of $F^\times$ with $e(\tau^\gamma)=t^\gamma$, and the resulting map $\gamma\mapsto \tau^\gamma: \Gamma\to F^\times$ has the following properties:  for all $\alpha, \beta, \gamma$,
\begin{enumerate}
\item[(T7)] $\tau^{\alpha+\beta}=\tau^\alpha \tau^\beta$;
%\item[(T8)]  $v_E(\tau^\gamma)=\gamma$;
\item[(T8)]  $\sp(\tau^\gamma)=\{\gamma\}$.
\end{enumerate} 

\noindent
Except for (T5), these properties are logically of first-order nature in terms of $F$ with this truncation operation
$\Gamma\times F\to F$ and the map $\gamma\mapsto \tau^\gamma: \Gamma\to F^\times$.

Define a {\em truncation structure on $F$} to be a map $(f,\alpha)\mapsto f|_{\alpha}: F\times \Gamma\to F$
together with a map $\gamma\mapsto \tau^\gamma: \Gamma \to F^\times$ such that (T1)--(T8) are satisfied. 
We just saw that any embedding $F\to C((t^\Gamma))$ with truncation closed image induces a truncation structure on $F$. Our aim is to reverse this:

\begin{theorem}\label{t1}  Any truncation structure on $F$ as above is induced by a unique embedding  $F\to C((t^\Gamma))$ sending each $\tau^\gamma$ to $t^\gamma$, with truncation closed image.
\end{theorem} 

\noindent
In the next section we prove some lemmas in the more general setting of valued vector spaces over $C$. 
In the last section we prove the theorem above.

 I thank the referees for their comments. 

\section{Truncation in Hahn spaces} 

\medskip\noindent
In this section $C$ is a field and $E$ a vector space over $C$. 
A {\em valuation\/} on $E$ is a surjective map
$v: E \to \Gamma_{\infty}$ such that for all $f,g\in E$ and $c\in C^\times$ we have
$v(f)=\infty\Leftrightarrow f=0$, $v(cf)= v(f)$, and $v(f+g)\ge \min(vf, vg)$; here and in the rest of this section
 $\Gamma$ is a linearly ordered set, and $\Gamma_{\infty}:= \Gamma\cup\{\infty\}$ where $\infty\notin \Gamma$, and the linear ordering of $\Gamma$ is extended to a linear ordering of $\Gamma_{\infty}$ by requiring
 that $\gamma< \infty$ for all $\gamma\in \Gamma$. We let $f,g$ range over $E$ and $\alpha, \beta, \gamma$ over $\Gamma$.

\medskip\noindent
Let now $E=(E,\Gamma, v)$ be a {\em valued\/}  vector space, that is, the vector space $E$ over $C$ comes equipped with a valuation $v:  E \to \Gamma_{\infty}$. We set 
\begin{align*}  f\preceq g\ &:\Leftrightarrow\ vf\ge vg,\qquad f\asymp g\ :\Leftrightarrow\ vf=vg,\\
f\prec g\ &:\Leftrightarrow\ vf>vg,\qquad f\sim g\ :\Leftrightarrow\ f-g\prec f.
\end{align*}
so $\asymp$ and $\sim$ are equivalence relations on $E$ and $E\setminus \{0\}$, respectively, and if
$f\sim g$, then $f\asymp g$. 
From \cite[2.3]{ADH} we recall that $E$ is said to be a {\em Hahn space} if for all $f,g\ne 0$  with
$f\asymp g$ there is a $c\in C^\times$ such that $f\sim cg$. 

The Hahn space $C[[t^\Gamma]]$ will play a special role:
except for the lack of a product operation it is defined just as the Hahn field $C((t^\Gamma))$; in particular, the wellordered subset $\supp f$ of $\Gamma$ and $v_t(f)\in \Gamma_{\infty}$ for $f\in C[[t^\Gamma]]$ are defined in the same way.

\bigskip\noindent
Let a map $(f,\gamma)\mapsto f|_{\gamma} :  E\times \Gamma \to E$ 
be given satisfying the analogues of (T1)--(T4), that is,
 for all $f,g$, all $c\in C$, and all $\alpha, \beta$:\begin{enumerate}
\item[(t1)] $v\big(f-(f|_{\alpha})\big)\ge \alpha$;
\item[(t2)] $v(f)\ge \alpha\ \Rightarrow\ f|_{\alpha}=0$;
\item[(t3)] $\beta> \alpha\ \Rightarrow\ (f|_{\alpha})|_{\beta}=f|_{\alpha}$;
\item [(t4)]$(f+g)|_{\alpha}=(f|_{\alpha})+(g|_{\alpha})$ and $(cf)|_{\alpha}=c(f|_{\alpha})$;
%\item $\Gamma(f):=\{\gamma\in \Gamma:\ v\big(f-(f|_{\gamma})\big)=\gamma\}$ is wellordered.
\end{enumerate} 

\noindent
The obvious truncation map on the Hahn space 
$E=C[[t^\Gamma]]$ over $C$ satisfies (1)--(4). 

\medskip\noindent
What conditions on the map
$(f,\gamma) \mapsto f|_{\gamma}$ yield the existence of an embedding $e: E \to C[[t^\Gamma]]$ of vector spaces over $C$ such that $v_t\big(e(f)\big)=v(f)$ and $e(f|_{\gamma})=e(f)|_{\gamma}$ for all $f$ and $\gamma$? An obviously necessary condition is that $E$ is a Hahn space. 

We set $\sp(f):= \{\gamma:\ v\big(f-(f|_{\gamma})\big)=\gamma\}$. Note that for the truncation map on the Hahn space $E=C[[t^\Gamma]]$ we have $\sp(f)=\supp(f)$. Thus another necessary condition for the existence of an embedding
$E\to C[[t^\Gamma]]$ as above is that $\sp(f)$ is wellordered for all $f$. Proposition~\ref{emb} below says that these two necessary conditions are together also sufficient.

\begin{lemma}\label{h1} For all $f,g, \alpha, \beta, \gamma$: \begin{enumerate}
\item[(i)] if $f|_{\gamma}\ne 0$, then $f\sim f|_{\gamma}$;
\item[(ii)]  if $\beta\leqslant \alpha$, then $(f|_{\alpha})|_{\beta}=f|_{\beta}$;
\item[(iii)] if $vf=\alpha$, then $\alpha= \min \sp(f)$; in particular, if $f\ne 0$, then $\sp(f)\ne \emptyset$.
\end{enumerate}
\end{lemma}
\begin{proof} (i) is clear from (t1) and (t2). Suppose $\beta\leqslant \alpha$.
Then $v(f-f|\alpha)\geqslant \alpha\geqslant \beta$, so by (t2) we have
$(f-f|_\alpha)|_\beta=0$, that is, $(f|_{\alpha})|_{\beta}=f|_{\beta}$. As to (iii), 
suppose $v(f)=\alpha$. Then $\sp(f)\geqslant \alpha$ by (i), and $f|_{\alpha}=0$ by (t2), so $\alpha\in \sp(f)$.
\end{proof}

\begin{lemma}\label{lemm2} The sets $\sp(f)$ have the following properties: \begin{enumerate}
\item[(i)] $\sp(f|_\alpha)=\sp(f)^{<\alpha}$;  
\item[(ii)] $\sp(f+g)\ \subseteq\  \sp(f)\cup \sp(g)$, and $\sp(cf)=\sp(f)$ for $c\in C^\times$;
\item[(iii)] $\sp(f-f|_{\alpha})=\sp(f)^{\ge \alpha}$;
\item[(iv)] if $v(f-f|_{\alpha})=\beta$, then $\beta\in \sp(f)$;
\item[(v)] if $\alpha> \sp(f)$, then $f|_{\alpha}=f$, and otherwise $f|_{\alpha}=f|_{\beta}$ for a unique $\beta\in \sp(f)$ $($and $\alpha\le \beta$ for this $\beta)$. 
\end{enumerate}
\end{lemma}
\begin{proof} If  $\beta \geqslant \alpha$,
then $(f|_{\alpha})|_{\beta}=f|_{\alpha}$ by (t3) and Lemma~\ref{h1}(ii), so $\beta\notin \sp(f|_{\alpha})$.
Next, let $\beta\in \sp(f|_{\alpha})$. Then $\beta< \alpha$, so $(f|_{\alpha})|_{\beta}=f|_{\beta}$ by Lemma~\ref{h1}(ii), hence 
$v(f|_{\alpha}-f|_{\beta}))=\beta$, which together with $v\big(f-(f|_{\alpha})\big)\ge \alpha$ yields
$v(f-f|_{\beta})=\beta$, so $\beta\in \sp(f)$. Reversing this argument yields $\sp(f)^{<\alpha}\subseteq \sp(f|_{\alpha})$. 

As to (ii), note that  if $v\big(f-(f|_{\alpha})\big)>\alpha$ and $v\big(g-(g|_{\alpha})\big)>\alpha$, then $$v\big((f+g)-(f+g)|_{\alpha}\big)\ >\ \alpha.$$ 
As to (iii), $\sp(f-f|_{\alpha})\ge \alpha$ by (t1) and Lemma~\ref{h1}(iii), hence $\sp(f-f|_{\alpha})\subseteq \sp(f)^{\ge \alpha}$ by (i), (ii), (iii). 
It remains to note that $\sp(f)\subseteq \sp(f|_{\alpha})\cup \sp(f-f|_{\alpha})$ by (ii).
For (iv), assume $v(f-f|_{\alpha})=\beta$. Now $v(f-f|_{\beta})\ge\beta$, hence
$v(f|_{\alpha}-f|_{\beta})\ge\beta$, and so $\big(f|_{\alpha}-f|_{\beta}\big)|_{\beta}=0$ by (t2), that is,
$f|_{\alpha}-f|_{\beta}=0$ by (t3) , (t4), and Lemma~\ref{h1}(ii), and thus $v(f-f|_{\beta})=\beta$, in other words, $\beta\in \sp(f)$.
As to (v), from (iii) we obtain: $\alpha>\sp(f)\ \Leftrightarrow\ f|_{\alpha}=f$. Suppose $f|_{\alpha}\ne f$. Set $\beta:=v(f-f|_{\alpha})$. Then $\beta\in \sp(f)$
by (iv), whose proof also gives $f|_{\alpha}=f|_{\beta}$. Uniqueness is clear. 
\end{proof} 
 
\noindent 
Set $P=\{f:\ \sp(f) \text{ is a singleton}\}$. Then $C^\times P\subseteq P$, and if
$f\in P$, then $f|_{\alpha}=f$ for $\alpha>vf$ by Lemma~\ref{lemm2}(iii), and $f|_{\alpha}=0$ for $\alpha\le vf$ by (t2). 

\begin{lemma} \label{p} The set $P$ has the following properties: \begin{enumerate}
%\item for each $\alpha$ there exists an $f\in P$ with $v(f)=\alpha$;
\item[(i)] if $\sp(f)=\{\gamma_1,\dots, \gamma_n\}$ with $\gamma_1<\cdots < \gamma_n$, then there are $f_i\in P$ with $vf_i=\gamma_i$ for $i=1,\dots,n$, such that $f=f_1+\cdots + f_n$; 
\item[(ii)] $\sp(f)$ is finite iff $f\in \sum P:= \{f_1+\cdots + f_n:\ f_1,\dots, f_n\in P\}$;
\item[(iii)] if $f, g\in P\cup\{0\}$ and $f\asymp g$, then $f+g\in P\cup\{0\}$;
\item[(iv)] if $E$ is a Hahn space and $f,g\in P$, $f\asymp g$, then $f=cg$ for some $c\in C^\times$.
\end{enumerate}
\end{lemma} 
\begin{proof} Item (i) is obviously true for $n=0$ and $n=1$. Let $n>1$ with $\sp(f)$ as in the hypothesis of (i). 
Then $\sp(f|_{\gamma_2})=\{\gamma_1\}$ by Lemma~\ref{lemm2}(i). We set $f_1:=f|_{\gamma_2}$.
Then $f=f_1+f^*$ with $\sp(f^*)=\{\gamma_2,\dots, \gamma_n\}$ by Lemma~\ref{lemm2}(iii).
Assuming inductively that $f^*=f_2+\cdots + f_n$ with $f_i\in P$ and $vf_i=\gamma_i$ for $i=2,\dots,n$,
we obtain the desired result for $f$. 

Item (ii) follows from (i) using Lemma~\ref{lemm2}(ii). Lemma~\ref{lemm2} also yields (iii). 

Suppose $E$ is a Hahn space and $f,g\in P,\ v(f)=v(g)=\alpha$. Take $c\in C^\times$ such that $f\sim cg$. Then
$v(f-cg)>\alpha$, so $\sp(f-cg)>\alpha$, but also $\sp(f-cg)\subseteq \{\alpha\}$, so $\sp(f-cg)=\emptyset$, and thus
$f=cg$. 
\end{proof}

\noindent
Suppose now that $\sp(f)$ is wellordered for all $f$. Then there is for each $\alpha$ an $f\in P$ with $vf=\alpha$.
(Proof: Take any $f$ with $vf=\alpha$, and suppose $f\notin P$. Then we have $\beta>\alpha$ such that $(\alpha,\beta)\cap \sp(f)=\emptyset$, and so $f|_\beta\in P$ and $v(f|_{\beta})=\alpha$.)  

The {\em leading term $\fd(f)\in P\cup \{0\}$} of $f$ is  defined as follows:
$\fd(f):=0$ if $f=0$; $\fd(f):=f$ if $f\in P$. Let $f\notin P$, $vf=\alpha$. Then
$\fd(f):=f|_\beta\in P$ if $\beta>\alpha$ and $(\alpha,\beta)\cap \sp(f)=\emptyset$; this is unambiguous: if
$\alpha < \beta < \gamma$ and $(\alpha,\gamma)\cap \sp(f)=\emptyset$, then $f|_{\beta}-f|_{\gamma}=(f-f|_\gamma)|_{\beta}$ and $v\big(f-f|_{\gamma}\big) \ge \beta$, so $f|_{\beta}=f|_{\gamma}$ by (t2).

Note that for $f\ne 0$ we have $f\sim \fd(f)$ by Lemma~\ref{h1}(i).  

\medskip\noindent
We define the
$\gamma$-term $f\gamma\in P\cup \{0\}$ of $f$ as follows: if $\gamma\notin \sp(f)$, then $f{\gamma}:=0$, and if
$\gamma\in \sp(f)$, then $f{\gamma}:=\fd(f-f|_{\gamma})$. Some easily established facts: \begin{itemize}
\item if $\gamma\in \sp(f)$, then $v(f\gamma)=\gamma$;
\item $v(f-f|_\gamma - f\gamma)>\gamma$; 
\item if $v(f-f|_{\gamma}-g)>\gamma$, $\gamma\in \sp(f)$, and $g\in P\cup\{0\}$, then $g=f\gamma$;
\item $(f|_\alpha){\gamma}=f{\gamma}$ if $\gamma<\alpha$ and $(f|_\alpha){\gamma}=0$ if $\gamma\ge \alpha$;
\item $(f+g)\gamma=f\gamma + g\gamma$, and $(cf)\gamma=c\cdot f{\gamma}$;
\item the (additive) map $f\mapsto (f\gamma): E\to  E^\Gamma$ is injective. 
\end{itemize}  
For the crucial fifth item, use the first three items and Lemma~\ref{p}(iii). The last item follows from the first and the fifth.

\begin{prop} \label{emb} Suppose $E$ is a Hahn space and $\sp(f)$ is wellordered for all $f$. Then there exists an embedding
$e: E \to C[[t^\Gamma]]$ of vector spaces over $C$ such that $v_t\big(e(f)\big)=v(f)$ and $e(f|_{\alpha})=e(f)|_{\alpha}$ for all $f$ and $\alpha$.
\end{prop} 
\begin{proof} For each $\gamma$ we take an element $p_{\gamma}\in P$ with $v(p_\gamma)=\gamma$. 
Then $P\cup \{0\}$ is the set of all products $cp_{\gamma}$ with $c\in C$, by  Lemma~\ref{p}(iv).
We define a bijection $e_P: P\cup\{0\}\to Ct^\Gamma$ by $e_P(cp_\gamma)=ct^\gamma$ for $c\in C$.
Then $e_P(cf)=ce_P(f)$ and $e_P(f+g)=e_P(f)+e_P(g)$ for $f,g\in Cp_{\gamma}$, $c\in C$. We
extend $e_P$ to a map 
$$e\ :\  E \to C[[t^\Gamma]], \qquad e(f)\ :=\  \sum_{\gamma} e_P(f\gamma).$$
The last three facts  about $\gamma$-terms from the list above show that $e$ is an embedding of valued vector spaces over $C$ with $e(f|_{\alpha})=e(f)|_{\alpha}$ for all $f$ and $\alpha$.
\end{proof} 

\noindent
Let $E$ and $e$ be as in the proposition. Then $\sp(f)=\supp e(f)$, and $e$ maps $P\cup \{0\}$ bijectively onto $C t^\Gamma$. Transfinite induction on the ordinals of the wellordered sets $\sp(f)$ shows that $e$ is uniquely determined by its restriction to $P$.

\section{Proof of the Theorem} 

\noindent
Let the field $F$ with the valuation $v: F \to \Gamma\cup\{\infty\}$ and the lift $C$ of the residue field be as in the Introduction. We assume given a map
$$(f,\alpha)\mapsto f|_{\alpha}\ :\ F\times \Gamma\to F$$
satisfying (T1)--(T5) for now. (Later we add (T6)--(T8).) Throughout $f$ and $g$ range over $F$ and $\alpha,\beta, \gamma$ over $\Gamma$. The previous section yields the set
$P\subseteq F$, and for all $f$ and $\gamma$ the leading term $\fd(f)$ of $f$ and the $\gamma$-term $f\gamma$ of $f$. 

\begin{lemma}\label{FP} Suppose that that for all $f\in P$ and all $g,\gamma$ with $\sp(g)<\gamma$
we have $\sp(fg)< vf + \gamma$. Then: \begin{enumerate}
\item[(i)] $P$ is a subgroup of $F^\times$;
\item[(ii)] for all $f\in P$ and all $g$ we have $\sp(fg) = vf +\sp(g)$.
\end{enumerate} 
\end{lemma} 
\begin{proof} For (i), let $f,g\in P$, set $vf=\alpha$, $vg=\beta$. Then 
 $v(fg)=\alpha+\beta= \min \sp(fg)$. Also $\sp(fg)< \alpha+\gamma$ for all $\gamma>\beta$, so $\sp(fg)=\{\alpha+\beta\}$, and thus $fg\in P$. This proves $PP\subseteq P$. Also $f=f\fd 1+fh$ with $h\prec 1$ and
$f\asymp f \fd 1\in P$, hence $h=0$, so $1=\fd 1\in P$.  To finish the proof of (i) we show $f^{-1}\in P$. We have
$f^{-1}=\fd(f^{-1})+h$ with $h\prec f^{-1}$, so $1=f\fd(f^{-1}) + fh$ with $1\in P$, $f\fd(f^{-1})\in P$, $1\asymp f\fd(f^{-1})$, and $fh\prec 1$, hence $fh=0$ and thus $f^{-1}=\fd(f^{-1})\in P$. 
 
Next, let $f\in P$ and let $\mu$ be the ordinal that corresponds to the wellordered set $\sp(g)$.
We prove (ii) by transfinite induction on $\mu$. If $\mu=0$, then $g=0$, and (ii) holds trivially.
 Next,  let $\mu \ge 1$ and $\beta\in \sp(g)$. Then
$g=g|_{\beta}+ g\beta + h$ with $vh> \beta$, so
$fg= f\cdot (g|_{\beta})+ f\cdot g\beta + fh$. Using an inductive assumption for the first equality:
$$\sp\big(f\cdot (g|_{\beta})\big)\ =\ \alpha + \sp(g)^{<\beta},\qquad \sp\big(f\cdot g\beta\big)\ =\ \alpha+\beta,\qquad \sp\big(fh\big)\ >\ \alpha+\beta.$$
It follows that $\sp(fg)^{<\alpha+\beta}=\sp\big(f\cdot (g|_{\beta})\big) = \alpha + \sp(g)^{<\beta}$ and
$$\sp(fg)^{=(\alpha+\beta)}\ =\ \sp\big(f\cdot g\beta\big)\ =\ \alpha+\beta,\qquad \sp(fg)^{>\alpha+\beta}\ =\ \sp(fh)\ >\ \alpha+\beta.$$ 
 Hence $\sp(fg)^{\le\alpha+\beta}=\alpha+\sp(g)^{\le\beta}$, and so 
$\sp(fg)\supseteq \alpha +\sp(g)$, by varying $\beta$.  If $\mu$ is a successor ordinal,
then for $\beta=\max \sp g$ we have $h=0$,  so $\sp(fg) =\alpha  +\sp(g)$, and (ii) holds.   

Now assume $\mu$ is a limit ordinal. 
For all $\gamma>\sp(g)$ we have $\sp(fg)< \alpha+ \gamma$. Let $\delta\in \sp(fg)$; then $\delta< \alpha+\gamma$ for all 
$\gamma>\sp(g)$, hence $\delta\le\alpha+\beta$ for some $\beta\in \sp(g)$ (otherwise,
$\gamma:= \delta-\alpha> \sp(g)$, a contradiction). Thus $\sp(fg)\subseteq \alpha+ \sp(g)$. 
\end{proof}

\noindent
\begin{lemma} \label{T6} Assume {\rm(T6)}:  for all $f,g,\gamma$, if $\sp(f)+\sp(g)<\gamma $, then $ \sp(fg)<\gamma$. 
Then for all $f,g$, $\sp(fg)\subseteq \sp(f)+\sp(g)$ . 
\end{lemma} 
\begin{proof} It follows from (T6) that the hypothesis of Lemma~\ref{FP} is satisfied. 
Let $\lambda$ and $\mu$ be the ordinals corresponding respectively to the wellordered sets $\sp(f)$ and $\sp(g)$.
We proceed by induction on the lexicographically ordered pair $(\lambda,\mu)$. By (ii) of Lemma~\ref{FP}
we have $\sp(fg)= \sp(f)+\sp(g)$ for $\lambda\le 1$ and also for $\mu\le 1$. For $\lambda$ a successor ordinal, $\alpha:=\max \sp(f)$ gives $fg=(f|_{\alpha})g+(f\alpha)g$, and inductively,
$\sp\big((f|_{\alpha})g\big)\subseteq \sp(f|_{\alpha})+\sp(g)$, and by (ii) of Lemma~\ref{FP}, $\sp\big((f\alpha) g\big)=\alpha+\sp(g)$, so
$\sp(fg)\subseteq \sp(f)+\sp(g)$. The same follows if $\mu$ is a successor ordinal. Now assume $\lambda$ and $\mu$ are limit ordinals $>0$, and $\gamma\in \sp(fg)$. Then the hypothesis of the lemma yields $\alpha\in \sp(f)$ and
$\beta\in \sp(g)$ such that $\gamma<\alpha+\beta$. Set $f_{\ge \alpha} :=  f-f|_{\alpha}$
and $g_{\ge \beta} := g-g|_{\beta}$. Then $fg=f\cdot g|_{\beta}+f|_{\alpha}\cdot g_{\ge \beta}+f_{\ge \alpha}\cdot g_{\ge \beta}$. Now inductively,  $\sp\big(f\cdot g|_{\beta}+f|_{\alpha}\cdot g_{\ge \beta}\big) \subseteq \sp(f)+\sp(g)$. Also
$\sp( f_{\ge \alpha}\cdot g_{\ge \beta})\ge\alpha+\beta$. Hence $\gamma\in \sp(f)+\sp(g)$.  
\end{proof}

\noindent
Note that for wellordered sets $A, B\subseteq \Gamma$ and any $\gamma$ there are only finitely many pairs $(\alpha,\beta)\in A\times B$ such that $\alpha+\beta=\gamma$.

\begin{lemma}\label{FP6} Assume {\rm(T6)}. Then $(fg)\gamma=\sum (f\alpha)(g\beta)$
where the summation is over the finitely many $(\alpha,\beta)\in \sp(f)\times \sp(g)$ with $\alpha+\beta=\gamma$.
\end{lemma} 
\begin{proof}
Take finite sets $A\subseteq \sp(f)$ and $B\subseteq \sp(g)$
such that $(\alpha,\beta)\in A\times B$ whenever $\alpha+\beta=\gamma$ and $(\alpha,\beta)\in \sp(f)\times \sp(g)$.
Set $f^*:= f-\sum_{\alpha\in A} f\alpha$ and $g^*:=g-\sum_{\beta\in B}g\beta$.  Then $\sp\big(\sum_{\alpha\in A} f\alpha\big)= A$, $\sp\big(\sum_{\beta\in B} g\beta\big)= B$, and
\begin{align*} fg\ &=\ \sum_{(\alpha,\beta)\in A\times B} (f\alpha)(g\beta) + (\sum_{\alpha\in A}f\alpha) g^* + f^*\cdot g,\\\sp(f^*)\ &=\  \sp(f)\setminus A, \qquad \sp(g^*)\ =\ \sp(g)\setminus B.
\end{align*} 
From Lemma~\ref{T6} we obtain
 $\gamma\notin \sp\big( (\sum_{\alpha\in A}f\alpha) g^*\big)$ and $\gamma\notin \sp(f^*\cdot g)$. Hence
 $$(fg)\gamma\ =\ \sum_{(\alpha,\beta)\in A\times B} \big((f\alpha)(g\beta) \big)\gamma\ =\ \sum_{(\alpha,\beta)\in A\times B,\ \alpha+\beta=\gamma} (f\alpha)(g\beta).\qquad\qquad\quad  \qedhere$$
\end{proof}

\noindent
Assume (T6). Then $P$ is a subgroup of $F^\times$ by Lemma~\ref{FP}.  
The inclusion $C^\times\to P$ and the valuation restricted to $P$ yield an exact sequence 
$$1\to C^\times \to P \to \Gamma\to 0.$$
If $F$ is algebraically closed, then $P$ is divisible,  so this sequence splits. Moreover, as we indicated in the introduction, 
any embedding $F\to C((t^\Gamma))$ with truncation closed image yields a splitting $\gamma\mapsto \tau^\gamma: \Gamma\to P$.

Let $\gamma\to \tau^\gamma: \Gamma\to P$ be a splitting. Together with our map $(f,\alpha)\mapsto f|_{\alpha}$
this yields a truncation structure on $F$. Given $f$ we have
$f{\gamma}=c_{\gamma}\tau^\gamma$ with $c_\gamma\in C$. Now the proof of
Proposition~\ref{emb} and  Lemma~\ref{FP6} show that the map
$$f\mapsto \sum_{\gamma} c_\gamma t^\gamma\ :\ F\to C((t^\Gamma))$$
is an embedding with truncation closed image that induces the given truncation structure on $F$. 
The remark following the proof of Proposition~\ref{emb} also shows that this is the only such embedding that sends 
$\tau^\gamma$ to $t^\gamma$ for all $\gamma$. This concludes the proof of Theorem~\ref{t1}.

\section*{Required Declarations}

\medskip\noindent
{\bf Ethical Approval}: not applicable.

\medskip\noindent
{\bf Funding}: no funding received.

\medskip\noindent
{\bf Availability of data and materials}: not applicable.

\bibliographystyle{amsplain}

\end{document}